\documentclass{amsart}
\RequirePackage[utf8]{inputenc}

\usepackage{amsmath,amsthm,amsfonts,amssymb,amscd,amsbsy,dsfont,multirow,hyperref}
\usepackage[all]{xy}

\newcommand{\dd}{\mathrm{d}}
\newcommand{\id}{\operatorname{id}}
\newcommand{\vol}{\operatorname{vol}}
\newcommand{\jac}{\operatorname{Jac}}
\newcommand{\Aut}{\operatorname{Aut}}
\newcommand{\Iso}{\operatorname{Iso}}
\newcommand{\Symp}{\operatorname{Symp}}
\newcommand{\Met}{\operatorname{Met}}
\newcommand{\Ric}{\operatorname{Ric}}
\newcommand{\V}{\operatorname{Vol}}
\newcommand{\embsigma}{\mathrm{Emb}^\mathcal L(\Sigma,M)}
\newcommand{\lagsigma}{\mathcal L(\Sigma,M)}

\newcommand{\R}{\mathds R}
\newcommand{\Z}{\mathds Z}
\newcommand{\C}{\mathds C}
\renewcommand{\O}{\mathsf O}
\newcommand{\SU}{\mathsf{SU}}
\newcommand{\SO}{\mathsf{SO}}

\newcommand{\J}{\mathcal J(M,\omega)}
\newcommand{\M}{\mathrm{Met}(M,\omega)}

\hypersetup{
    pdftoolbar=true,
    pdfmenubar=true,
    pdffitwindow=false,
    pdfstartview={FitH},
    pdftitle={Equivariant deformations of Hamiltonian stationary Lagrangian submanifolds},
    pdfauthor={Renato G. Bettiol, Paolo Piccione and Bianca Santoro},
    pdfsubject={53C55, 53D12, 58D10, 58E15, 58E40}, 
    pdfkeywords={}
    pdfnewwindow=true,
    colorlinks=true, 
    linkcolor=blue,
    citecolor=blue,
    urlcolor=blue,
}

\newtheorem{theorem}{Theorem}[]
\newtheorem{lemma}[theorem]{Lemma}
\newtheorem{proposition}[theorem]{Proposition}
\newtheorem{corollary}[theorem]{Corollary}

\newtheorem*{theorem*}{Theorem}

\theoremstyle{definition}
\newtheorem{definition}[theorem]{Definition}

\theoremstyle{remark}

\newtheorem{example}[theorem]{Example}

\title[Equivariant deformations of Hamiltonian stationary Lagrangians]{Equivariant deformations of Hamiltonian stationary Lagrangian submanifolds}
\author[R. G. Bettiol]{Renato G. Bettiol}
\author[P. Piccione]{Paolo Piccione}
\author[B. Santoro]{Bianca Santoro}

\address{
\begin{tabular}{lll}
University of Notre Dame & &Universidade de S\~ao Paulo \\
Department of Mathematics & & Departamento de Matem\'atica \\
255 Hurley Building & & Rua do Mat\~ao, 1010 \\
Notre Dame, IN, 46556-4618, USA & & S\~ao Paulo, SP, 05508-090, Brazil\\
\emph{E-mail address}: {\tt rbettiol@nd.edu} & & \emph{E-mail address}: {\tt piccione@ime.usp.br}\\
\end{tabular}
\bigskip
\hfill\break\hfill\indent
\begin{tabular}{lll}
The City College of New York, CUNY&&\\
Mathematics Department&&\\
138th St and Convent Avenue, NAC 4/112B&&\\
New York, NY, 10031, USA&&\\
\emph{E-mail address}: {\tt bsantoro@ccny.cuny.edu} & &
\end{tabular}
}

\allowdisplaybreaks
\numberwithin{equation}{section}
\numberwithin{theorem}{section}

\thanks{The first named author is partially supported by the NSF grant DMS-0941615, USA. The second named author is partially supported by Fapesp and CNPq, Brazil. The third named author is partially supported by the NSF grant DMS-1007155 and PSC-CUNY grants, USA}

\date{\today}

\begin{document}
\begin{abstract}
We prove an equivariant deformation result for Hamiltonian stationary Lagrangian submanifolds of a K\"ahler manifold, with respect to deformations of its metric and almost complex structure that are compatible with an isometric Hamiltonian group action. This yields existence of Hamiltonian stationary Lagrangian submanifolds in possibly non-K\"ahler symplectic manifolds whose metric is arbitrarily close to a K\"ahler metric.
\end{abstract}

\maketitle

\section{Introduction}

Let $(M,\omega)$ be a symplectic manifold with a Riemannian metric $g$. A submanifold $\Sigma$ of $M$ with $\dim M=2\dim \Sigma$ is \emph{Lagrangian} if the restriction $\omega|_\Sigma$ of the symplectic form to this submanifold vanishes identically. Multiple constrained variational problems related to minimizing volume of Lagrangian submanifolds have been extensively studied in the literature, see for instance \cite{ohnita,ohnita2,bc,jls,lee,oh1,oh2,sw}. Among these, an important constrained variational problem is related to minimizing volume under \emph{Hamiltonian variations}. The corresponding critical points, called \emph{Hamiltonian stationary Lagrangian submanifolds}, and their possible deformations in an equivariant setup, are the main objects of study in this paper.

Denote by $\lagsigma$ the space of Lagrangian submanifolds of $M$ that are diffeomorphic to $\Sigma$.
A Hamiltonian variation of a Lagrangian submanifold $\Sigma$ is simply a variation by a Hamiltonian vector field $X$, i.e., $X$ is a vector field on $M$ along $\Sigma$, such that the $1$-form $\omega(X,\cdot)|_{\Sigma}$ is exact. In this situation, the submanifolds $\Sigma_t:=\{\exp_p(tX_p):p\in\Sigma\}$, $|t|<\varepsilon$, are also Lagrangian, i.e., $\Sigma_t\in\lagsigma$ is a curve through $\Sigma=\Sigma_0$; but generally not all Lagrangians near $\Sigma$ are obtained this way, see Section~\ref{sec:stationarysubmflds}. More precisely, this curve $\Sigma_t$ of Lagrangians stays inside the integral leaf through $\Sigma$ of a certain (integrable) distribution of $\lagsigma$, with codimension $b_1(\Sigma)$, that we call the \emph{Hamiltonian distribution}, see Subsection~\ref{subsec:unpar}. The integral leaves of this distribution are locally parametrized by the first de Rham cohomology $H^1(\Sigma,\mathds R)$, which is a real vector space of dimension $b_1(\Sigma)$. Given a closed $1$-form $\eta$ on $\Sigma$, we denote by $[\eta]\in H^1(\Sigma,\mathds R)$ its cohomology class and by $\lagsigma_{[\eta]}$ the integral leaf of the Hamiltonian distribution that corresponds to $[\eta]$. When $\eta$ is exact, i.e., $[\eta]=0$, then the above means that $\lagsigma_{[\eta]}$ is the space of Hamiltonian variations of $\Sigma$. A Lagrangian submanifold $\Sigma\subset M$ is a $g$-\emph{Hamiltonian stationary Lagrangian submanifold} if it has critical volume (with respect to the volume form induced by $g$) among all its Hamiltonian variations.

Hamiltonian stationary Lagrangian submanifolds have been used, among others, to provide canonical representatives of the Lagrangian homology (the part of the homology generated by Lagrangian cycles), see \cite{sw}. Issues related to their existence in various contexts were discussed in \cite{bc,jls,lee,oh2}, and questions regarding their stability were addressed in \cite{ohnita,ohnita2,oh1}. In the present paper, we are interested in an \emph{equivariant} rigidity
 notion, that allows to deform $g$-Hamiltonian stationary Lagrangian submanifolds according to a suitable deformation of the metric $g$. This works in a similar fashion to the implicit function theorem, but taking into account the ambiguity
imposed by a group of symmetries of the variational problem. Namely, we assume that there is an isometric Hamiltonian action of a compact Lie group $G$ on $M$. Such an action carries Lagrangian submanifolds to Lagrangian submanifolds (since the action is by symplectomorphisms), and preserves the volume functional (since the action is by isometries). In this way, a Hamiltonian stationary Lagrangian submanifold $\Sigma$ is automatically degenerate due to the presence of action fields, and an appropriate \emph{equivariant $G$-nondegeneracy} condition is introduced, see Definition~\ref{def:nondeg}. Our main result concerns deformations of such equivariantly nondegenerate submanifolds, corresponding to deformations of the metric $g$ (or of the associated almost complex structure $J$) that
preserve the group of symmetries $G$. More precisely, we prove the following:

\begin{theorem*}
Let $(M,\omega,g_0,J_0)$ be a K\"ahler manifold with an isometric Hamiltonian action of a compact Lie group $G$.
Suppose that  either:
\begin{itemize}
\item[(A)] There exists a smooth deformation $[-\delta,\delta]\ni t\mapsto g_t\in\Met(M)$ of the metric $g_0$, such that $G$ acts by $g_t$-isometries for all $t\in [-\delta,\delta]$; and let $J_t:=J_{g_t}$ be the corresponding family of $\omega$-compatible almost complex structures;\footnote{See \eqref{eq:pj} and Corollary \ref{cor:deform}.} 
\end{itemize}
or,
\begin{itemize}
\item[(B)] There exists a smooth deformation $[-\delta,\delta]\ni t\mapsto J_t\in\J$ of $J_0$ by $\omega$-compatible almost complex structures, such that $G$ acts by $J_t$-biho\-lo\-morphisms for all $t\in [-\delta,\delta]$; and let $g_t(\cdot,\cdot):=\omega(\cdot,J_t\cdot)$ be the corresponding family of Riemannian metrics.
\end{itemize}
Suppose $\Sigma_0\subset (M,\omega)$ is a $G$-nondegenerate $g_0$-Hamiltonian stationary Lagrangian submanifold.
Then there exists $\varepsilon>0$, a neighborhood $\mathcal V$ of $\Sigma_0\in\lagsigma$, a neighborhood $\mathcal E$ of $[0]\in H^1(\Sigma,\mathds R)$  and a
smooth map $\Sigma\colon\, (-\varepsilon,\varepsilon)\times\mathcal E\,\to\mathcal V$; such that $\Sigma_{t,[\eta]}:=\Sigma(t,[\eta])$ is a $g_t$-Hamiltonian stationary Lagrangian submanifold in $\lagsigma_{[\eta]}$ for all $|t|<\varepsilon$  and all $[\eta]\in\mathcal E$,
and $\Sigma(0,[0])=\Sigma_0$. Moreover, if $\Sigma_*\in\mathcal V$ is a $g_t$-Hamiltonian stationary Lagrangian submanifold
in $\lagsigma_{[\eta]}$ sufficiently close to the $G$-orbit of $\Sigma_0$, then there exists $\phi\in G$ such that $\phi(\Sigma_*)=\Sigma_{t,[\eta]}$.
\end{theorem*}

Note that the manifolds $(M,\omega,g_t,J_t)$ might not be K\"ahler for $t\neq 0$; the only requirement is that both the metric $g_t$ and the almost complex structure $J_t$ be compatible with the fixed symplectic form $\omega$. In particular, this result abstractly yields existence of $g$-Hamiltonian stationary Lagrangian submanifolds in certain symplectic manifolds equipped with a metric $g$ that is a small deformation of a K\"ahler metric, see \cite{bc,jls,lee}.

The main ingredients in the proof of the above Theorem are the appropriate variational formulation of the problem, which is cast in (quotients of) H\"older spaces due to Fredholmness reasons and the equivariant implicit function theorem with low regularity studied in \cite{bps}. The latter is an abstract equivariant formulation of the classic implicit function theorem in a low regularity setup tailored to geometric variational problems. Among the crucial hypotheses are that the linear operator that represents the second variation of the functional in question be a Fredholm operator of index zero. This follows easily in the case of Hamiltonian stationary Lagrangians using standard Schauder estimates, since the corresponding operator is a (fourth order) elliptic operator.

This paper is organized as follows. In Section~\ref{sec:basicstuff}, we recall basic concepts in symplectic geometry and various Hamiltonian constructions and deformations preserving their symmetry groups. The main aspects of the constrained variational problem of Hamiltonian stationary Lagrangian submanifolds are studied in Section~\ref{sec:stationarysubmflds}, where we also recall the first and second variations in the K\"ahler case. The rigorous framework for the proof of the above Theorem is discussed in Section~\ref{sec:proofmain}. Finally, Section~\ref{sec:examples} contains a few examples of deformations to which this result applies.

\smallskip
\noindent
{\bf Acknowledgement.} It is our pleasure to thank Andr\'e Carneiro, Richard Hind and Tommaso Pacini for valuable suggestions. We also acknowledge hurricane Sandy, that despite having many victims and causing terrible destruction, at the same time forced the first and third named authors to be stranded for several days while no transportation was
 available, providing a good amount of time to exchange ideas and start a collaboration whose first results led to this note.

\section{Preliminaries}
\label{sec:basicstuff}

In this first section, we recall a few basic facts about the interplay of the isomorphism groups of symplectic, almost complex and Riemannian structures and basic definitions regarding Hamiltonian actions. The reader with a working knowledge of such material may proceed to Section~\ref{sec:stationarysubmflds}.

\subsection{Compatible triples}
Given a (necessarily even dimensional) real vector space $V$, consider the following objects on $V$:
 \begin{itemize}
 \item[(i)] a nondegenerate skew-symmetric bilinear form $\omega\colon V\times V\to\R$;
 \item[(ii)] a complex structure $J\colon V\to V$;
 \item[(iii)] a positive-definite inner product $g\colon V\times V\to\R$.
\end{itemize}
Denote by $\Symp(V,\omega)$ the group of symplectomorphisms of $(V,\omega)$, i.e., automorphisms $T\colon V\to V$ such that $\omega(T\cdot,T\cdot)=\omega(\cdot,\cdot)$. Denote by $\Aut(V,J)$ the group of \emph{$J$-biholomorphisms}, or automorphisms of $(V,J)$, i.e., automorphisms $T\colon V\to V$ that commute with $J$. Finally, denote by $\O(V,g)$ the group of $g$-orthogonal isomorphisms of $V$, i.e., automorphisms $T\colon V\to V$ such that $g(T\cdot,T\cdot)=g(\cdot,\cdot)$.

We say that $(\omega,J,g)$ is a \emph{compatible triple} if $\omega(\cdot,\cdot)=g(J\cdot,\cdot)$, or, equivalently, if $\omega(\cdot,J\cdot)=g(\cdot,\cdot)$. If $(\omega,J,g)$ is a compatible triple, then:
\begin{equation}\label{eq:groups-lin2}
\O(V,g)\cap\Aut(V,J)=\O(V,g)\cap\Symp(V,\omega)=\Aut(V,J)\cap\Symp(V,\omega).
\end{equation}
In other words, given vector spaces $V_i$ endowed with compatible triples $(\omega_i,J_i,g_i)$, $i=1,2$, an isomorphism $T\colon V_1\to V_2$ that preserves any two of the structures in the triple, automatically preserves the third one. In this way, choosing any two structures among $\{\omega,J,g\}$ on a vector space $V$ determines the third one, so that the triple $(\omega,J,g)$ is compatible.

The definition of compatible triples carries over naturally to (even dimensional) smooth manifolds $M$ endowed with
a symplectic form $\omega$, an almost complex structure $J$ and a Riemannian metric $g$. Namely, such a triple $(\omega,J,g)$ is \emph{compatible} if at every point $p\in M$, $(\omega_p,J_p,g_p)$ is a compatible triple on $T_pM$. In the special case where $J$ is an integrable almost complex structure (hence a complex structure), the manifold $M$ equipped with a compatible triple $(\omega,J,g)$ is called a \emph{K\"ahler manifold}.

Equation \eqref{eq:groups-lin2} implies that if $(\omega,J,g)$ is a compatible triple on $M$, then
\begin{multline}\label{eq:groups2}
\Iso(M,g)\cap\Symp(M,\omega)=\Iso(M,g)\cap\Aut(M,J)\\ =\Aut(M,J)\cap\Symp(M,\omega).
\end{multline}
Note that, $\Iso(M,g)$ is always a (finite-dimensional) Lie group, and it is compact when $M$ is compact. The groups $\Aut(M,J)$ and $\Symp(M,\omega)$ are infinite-dimensional, however both intersections $\Iso(M,g)\cap\Symp(M,\omega)$ and $\Iso(M,g)\cap\Aut(M,J)$ are subgroups of $\Iso(M,g)$ which are closed in the $C^1$-topology\footnote{Hence, they are also closed in the $C^0$-topology, since both topologies coincide on $\Iso(M,g_J)$}, and therefore are Lie subgroups of $\Iso(M,g)$.

\subsection{Basic Hamiltonian constructions}\label{subsec:basichamilt}
Suppose $(M,\omega)$ is a symplectic manifold and $X$ is a vector field on $M$. Contracting $X$ with the symplectic form $\omega$, we get a $1$-form on $M$ denoted $\iota_X \omega:=\omega(X,\cdot)$. For convenience, we also use the special notation
\begin{equation}
\sigma_X:=\iota_X \omega=\omega(X,\cdot).
\end{equation}
When $X$ is a field along a submanifold $\Sigma\subset M$, we also write $\sigma_X$ for the $1$-form on $\Sigma$ obtained by pulling back the contracted $1$-form $\iota_X \omega$, i.e., $\sigma_X=x^*(\iota_X \omega)$. As a word of caution, this $1$-form is unrelated to the $2$-form given by the pull-back $x^*\omega$.

A function $h\colon M\to\R$ is called a \emph{Hamiltonian function} or \emph{Hamiltonian potential} for the vector field $X$ on $M$ if
\begin{equation}\label{eq:hamilt}
\dd h=\sigma_X.
\end{equation}
In this case, $X$ is called the \emph{symplectic gradient} of $h$. As a side note, if $(M,\omega)$ is K\"ahler, then \eqref{eq:hamilt} is equivalent to $X=-J\nabla h$, where $\nabla h$ is the Riemannian gradient of $h$.

Vector fields on a symplectic manifold that admit a Hamiltonian potential are called \emph{Hamiltonian vector fields}. Equivalently, a vector field is Hamiltonian if $\sigma_X\in B^1(M)$ is an \emph{exact} $1$-form. Vector fields such that $\sigma_X\in Z^1(M)$ is a \emph{closed} $1$-form are called \emph{symplectic vector fields}. The justification for this name is that the flow of such a vector field preserves the symplectic form $\omega$, i.e., the Lie derivative $\mathcal L_X \omega$ vanishes if $X$ is symplectic. Evidently, Hamiltonian fields are always symplectic, and the obstruction for symplectic fields to be Hamiltonian is measured by the first de Rham cohomology $H^1(M,\R)=Z^1(M)/B^1(M)$.

Now, suppose that a Lie group $G$ acts by symplectomorphisms on $(M,\omega)$, i.e., $g\colon M\to M$ preserves $\omega$ for all $g\in G$. Denote by $\mathfrak g$ the Lie algebra of $G$ and by $\mathfrak g^*$ its dual. The $G$-action is said to be \emph{Hamiltonian} if there exists a map $\mu\colon M\to\mathfrak g^*$, called \emph{moment map}, such that
\begin{itemize}
\item[(i)] $\mu\colon M\to\mathfrak g^*$ is $G$-equivariant, where the $G$-action considered on $\mathfrak g^*$ is the coadjoint action;
\item[(ii)] For every $X\in \mathfrak g$, denote by $X^*_p:=\tfrac{\dd}{\dd s}\big(\exp(sX)\cdot p\big)\big|_{s=0}$ the induced action field on $M$.\footnote{i.e., $X^*$ is the vector field on $M$ that is the infinitesimal generator of the $1$-parameter group of diffeomorphisms of $M$ generated by the $1$-parameter subgroup $\R\ni s\mapsto\exp(sX)\in G$.} Then $\langle\mu(\cdot),X^*\rangle\colon M\to\R$ is a Hamiltonian potential for the vector field $X^*$, i.e., $\dd\langle\mu(\cdot),X^*\rangle=\sigma_{X^*}$, where, as above, $\sigma_{X^*}=\iota_{X^*} \omega=\omega(X^*,\cdot)$.
\end{itemize}
In other words, every action field is a Hamiltonian field and the moment map $\mu$ encodes all the corresponding Hamiltonian potentials.

\begin{example}
If $G$ is a closed connected Lie subgroup of $\mathsf U(n+1)$, then the restriction of the transitive $\mathsf U(n+1)$-action on $\C P^n$ to $G$ is a Hamiltonian action, with moment map $\mu([z])=\pi_{\mathfrak g}(-i\, zz^*/2\|z\|^2)$, where $\pi_{\mathfrak g}$ is the orthogonal projection onto $\mathfrak g$, with respect to an $\mathrm{Ad}$-invariant inner product in $\mathsf U(n+1)$.
\end{example}

\subsection{Deformations preserving symmetries}
Constructing compatible triples on a manifold is quite elementary. We are interested in a slightly more elaborate problem, namely that of constructing $1$-parameter families $(\omega,J_t,g_t)$ of compatible triples for a fixed symplectic form $\omega$, that preserve a nontrivial subgroup $G$ of \eqref{eq:groups2} that acts in a Hamiltonian way on $M$, and for $t=0$ turn $M$ into a K\"ahler manifold. In general, most deformations of this type do not produce other K\"ahler structures, i.e., $J_t$ is non-integrable for $t>0$, but this is not an issue for our applications. For an example in which integrability is preserved, see Subsection~\ref{subsec:ricciflow}.

We now observe that, due to \eqref{eq:groups2}, such deformations $(\omega,J_t,g_t)$ can be obtained by first considering a deformation $g_t$ of the metric $g_0$ that preserves the isometric $G$-action; and then considering the corresponding deformation $J_t$ of the almost complex structure $J_0$ with respect to the fixed symplectic form $\omega$, see Corollary \ref{cor:deform}. A few concrete constructions of deformations of compatible triples are described in Section \ref{sec:examples}. Let us give more details on how the above works.

An almost complex structure $J$ is called \emph{$\omega$-compatible} if the triple $(\omega,J,g_J)$, where $g_J(\cdot,\cdot):=\omega(\cdot,J\cdot)$, is a compatible triple. In other words, $J$ is $\omega$-compatible if $g_J$ is a Riemannian metric. Define the spaces:
\begin{equation*}
\begin{split}
\J &:=\big\{J\colon TM\to TM: J \mbox{ is an } \omega\mbox{-compatible almost complex structure}\big\};\\
\M&:=\big\{g\in\Met(M):g=g_J \text{ for some } J\in\J\big\}.
\end{split}
\end{equation*}
The map $\J\ni J\mapsto g_J\in\M$ is clearly a bijection, whose inverse will be denoted by $\M\ni g\mapsto J_g\in\J$. Let us recall the following standard result that, in particular, implies that $\M$ is homotopically equivalent to $\Met(M)$ hence contractible, see \cite[Prop 2.50, 2.51]{ms}.

\begin{proposition}\label{prop:gtoj}
Let $(M,\omega)$ be a symplectic manifold. There exists a smooth retraction $r\colon\Met(M)\to\M$.
\end{proposition}

\begin{proof}
Given $g\in\Met(M)$, there exists a unique skew-symmetric $(1,1)$-tensor $A_g$ on $M$ that satisfies
$\omega(\cdot,\cdot)=g(A_g\cdot,\cdot)$. Since $\omega$ is everywhere nondegenerate, $A_g$ is everywhere nonsingular. The pointwise polar decomposition of $A_g$ provides two unique $(1,1)$-tensors\footnote{Here, $T^*$ denotes the $g$-adjoint of a $(1,1)$-tensor $T$ on $M$, defined by $g(T^*\cdot,\cdot)=g(\cdot,T\cdot)$.} on $M$,
\begin{equation}\label{eq:pj}
P_g=(A_gA_g^*)^{\frac12} \quad \mbox{ and }\quad J_g=P_g^{-1}A_g,
\end{equation}
such that:
\begin{itemize}
\item[(i)] $P_gJ_g=J_gP_g=A_g$;
\item[(ii)] $P_g$ is positive, i.e., $g(P_g\cdot,\cdot)$ is symmetric and positive-definite;
\item[(iii)] $J_g$ is $g$-orthogonal, i.e., $J_g^*=J_g^{-1}$.
\end{itemize}
Since $A_g$ is skew-symmetric and $P_g$ is symmetric, then
\begin{equation*}
J_g^{-1}=J_g^*=A_g^*P_g^{-1}=-A_gP_g^{-1}=-J_g,
\end{equation*}
i.e., $J_g$ is an almost complex structure. The desired map $r$ is given by $r(g):=g(P_g\cdot,\cdot)$. Since $P_g=J_g^{-1}A_g=J_g^*A_g$, then
\begin{equation}\label{eq:defr(g)}
r(g)=g(P_g\cdot,\cdot)=g(J_g^*A_g\cdot,\cdot)=g(A_g\cdot,J_g\cdot)=\omega(\cdot, J_g\cdot)\in\M.
\end{equation}
By the uniqueness of the polar decomposition, it is also immediate to see that $r(g)=g$ if $g\in\M$, i.e.,
$r$ is a retraction. Smoothness follows immediately from the smoothness of the polar decomposition in the open set of invertible operators.
\end{proof}

As an immediate consequence of the above result and \eqref{eq:groups2}, we get a way of deforming compatible triples preserving their symmetry, whose sole input is a metric deformation preserving an isometric action. This process fits the set of hypotheses (A) in the Theorem in the Introduction. This choice of deformation intuitively allows more examples, since it is generally easier to deform a metric preserving a group action than deforming an almost complex structure preserving its automorphism group.

\begin{corollary}\label{cor:deform}
Let $(\omega,J_0,g_0)$ be a compatible triple on $M$, and suppose that $G$ is a Lie group that acts on $M$ by symplectomorphisms and $g_0$-isometries (hence by $J_0$-biholomorphisms). Assume $g_t$, $t\in [-\delta,\delta]$, is a deformation of $g_0$ such that the $G$-action is by $g_t$-isometries for $t\in [-\delta,\delta]$, and let $J_t:=J_{g_t}$ be the almost complex structure obtained from $g_t$ as in \eqref{eq:pj}. Then the triple $(\omega,J_t,g_t)$, $t\in [-\delta,\delta]$, is compatible and a deformation of $(\omega,J_0,g_0)$, so that $G$ acts by $g_t$-isometries and $J_t$-biholomorphisms, $t\in [-\delta,\delta]$.
\end{corollary}

\section{Hamiltonian stationary Lagrangian submanifolds}
\label{sec:stationarysubmflds}

\subsection{Lagrangian submanifolds}
\label{sub:LagSubmanifolds}
Let $(M,\omega)$ be a symplectic manifold and $\Sigma$ be a compact manifold with $\dim\Sigma=\tfrac12\dim M$. An embedding $x\colon\Sigma\hookrightarrow (M,\omega)$ is called \emph{Lagrangian} if $x^*\omega=0$. In this case, we say $x(\Sigma)\subset M$ is a Lagrangian submanifold. Lagrangian submanifolds play a central role in symplectic geometry, see \cite{ms} and references therein.
 A smooth family $x_s\colon\Sigma\to M$, $s\in\left(-\varepsilon,\varepsilon\right)$, of embeddings is called a \emph{Lagrangian} (respectively, \emph{Hamiltonian}) \emph{deformation of} $x_0=x$ if its derivative $X=\frac{\dd}{\dd s}x_s\big|_{s=0}$ is a symplectic (respectively, Hamiltonian) vector field along $x$, i.e., if the $1$-form $\sigma_X:=x^*(\omega(X,\cdot))$ on $\Sigma$ is \emph{closed} (respectively, \emph{exact}), see Subsection~\ref{subsec:basichamilt}.

\begin{example}\label{ex:sphere}
Consider the sphere $S^2=\C P^1$ with its standard K\"ahler structure. Any great circle $\Sigma$ on $S^2$ is Lagrangian (and minimal), and divides $S^2$ into two domains of same area. Any deformation of $\Sigma$ through other smooth curves is a Lagrangian deformation; however only those deformations through curves that still bisect the area of $S^2$ are Hamiltonian, by Stokes' Theorem.
\end{example}

A particularly interesting feature of a Lagrangian submanifold $\Sigma$ of a K\"ahler manifold $(M,\omega)$ is the following. Denote by $H$ the mean curvature vector of $\Sigma$ in $M$ and consider the contracted $1$-form $\sigma_H$. Then $\dd \sigma_H=\Ric|_\Sigma$, where $\Ric$ is the Ricci $2$-form on $M$. In particular, if $(M,\omega)$ is K\"ahler-Einstein, i.e., $\Ric=\kappa\,\omega$, then $\dd\sigma_H=0$ and thus $H$ is a symplectic vector field, i.e., an infinitesimal Lagrangian deformation. As observed by \cite{sw}, this suggests that it is natural to consider variational problems for volume of submanifolds with a Lagrangian constraint, as follows.

\subsection{Minimizing volume}
\label{subsec:minvol}

Let $(\omega,J,g)$ be a compatible triple on $M$ and consider the corresponding volume form $\vol_g$ on $M$. This provides a way to measure the volume of an embedding $x\colon\Sigma\to M$, by setting
\begin{equation}\label{eq:volg}
\V_g\big(x(\Sigma)\big)=\int_\Sigma x^*(\vol_{g}).
\end{equation}
A Lagrangian embedding $x_0\colon\Sigma\to M$ is called \emph{$g$-Lagrangian} (respectively, \emph{$g$-Hamiltonian}) \emph{stationary} if it has critical volume with respect to any Lagrangian (respectively, Hamiltonian) deformations $x_s\colon\Sigma\to M$, $s\in\left(-\varepsilon,\varepsilon\right)$, of $x_0$, i.e.,
\begin{equation}\label{eq:crit}
\frac{\dd}{\dd s}\V_g\big(x_s(\Sigma)\big)\Big|_{s=0}=\int_{\Sigma} g(H,X)=0,
\end{equation}
where $H$ is the mean curvature vector of $\Sigma$ in $(M,g)$, $X_p=\frac{\dd}{\dd s}x_s(p)\big|_{s=0}$ is the variation field and the integration is with respect to the pulled-back volume form $x^*(\vol_{g})$. In some references, Hamiltonian stationary Lagrangian submanifolds are also called \emph{$H$-minimal submanifolds}. This reflects the fact that they minimize the volume functional in some directions, the \emph{Hamiltonian} directions; while minimal submanifolds are critical points of the volume functional with respect to all directions (hence, in particular, are Hamiltonian stationary).

Notice that the above functional remains invariant if we replace $x(\Sigma)$ with $\phi(x(\Sigma))$, where $\phi\colon M\to M$ is an isometric symplectomorphism. Namely, $\phi(x(\Sigma))$ is still Lagrangian because $\phi$ is a symplectomorphism, and it has the same volume as $x(\Sigma)$ because $\phi$ is an isometry.

Let us derive the Euler-Lagrange equations for this functional. If $X$ is a Hamiltonian variation of $x_0$, then $\sigma_X=\dd h$, for some $h\colon\Sigma\to\R$. Thus, if $x_0$ is $g$-Hamiltonian stationary Lagrangian, then \eqref{eq:crit} reads
\begin{equation*}
0=\int_{\Sigma} g(H,X)=\int_{\Sigma} g(\sigma_H,\sigma_X)  =\int_{\Sigma} g(\sigma_H,\dd h)=\int_{\Sigma} (\delta\sigma_H)h,
\end{equation*}
where $\delta$ is the codifferential, i.e., the formal adjoint of the exterior derivative operator $\dd$. Since the above vanishes for all $h$, we get that the Euler-Lagrange equations for a Hamiltonian stationary Lagrangian embedding $x_0$ are
\begin{equation}\label{eq:el-kahler}
\delta\sigma_H=0.
\end{equation}
As pointed out before, if $(M,\omega,J)$ is K\"ahler-Einstein, then $\dd\sigma_H=0$. Consequently, in the K\"ahler-Einstein case, \eqref{eq:el-kahler} is equivalent to $\Delta \sigma_H=0$, where $\Delta=\dd\delta+\delta\dd$ is the (nonnegative) Laplace-de Rham operator; i.e., $\sigma_H$ is a harmonic $1$-form. In particular, it follows by Hodge theory that if $H^1(\Sigma,\R)=0$, then the Hamiltonian stationary Lagrangian embedding $x_0$ is actually minimal.

\begin{example}
An example of a Hamiltonian stationary Lagrangian submanifold that is not minimal is given by the standard tori $T=S^1(r_1)\times\dots\times S^1(r_n)\subset\C^n$. Note that $\C^n$ is K\"ahler-Einstein however the first de Rham cohomology of $T$ is not trivial. Another interesting family of examples of Hamiltonian stationary Lagrangians is given by curves with constant geodesic curvature on a Riemann surface.
\end{example}

\subsection{Second variation}

Following \cite[Def 2.6]{jls} and \cite[Thm 3.4]{oh2}, we now describe the second variation formula of the functional \eqref{eq:volg}, assuming that the compatible triple $(\omega,J,g)$ turns $M$ into a K\"ahler manifold. Given a $g$-Hamiltonian stationary Lagrangian embedding $x_0\colon\Sigma\to M$ and any Hamiltonian deformation $x_s$ of $x_0$,
\begin{multline*}
\frac{\dd^2}{\dd s^2}\V_g\big(x_s(\Sigma)\big)\Big|_{s=0}=\int_{\Sigma} \big[\, g(\Delta\dd h,\dd h)-\Ric(J\dd h,J\dd h)\\ -2g(\dd h\otimes\dd h\otimes\sigma_H,S)+g(\dd h,\sigma_H)^2\big],
\end{multline*}
where $h$ is a Hamiltonian potential for the variation $X_p=\frac{\dd}{\dd s}x_s(p)\big|_{s=0}$, i.e., $\sigma_X=\dd h$, and $S$ is a $(0,3)$-tensor on $\Sigma$ defined by $S(X,Y,Z)=g(J(B(X,Y)),Z)$, where $B$ is the second fundamental form of $x_0(\Sigma)\subset M$. In particular, we get an expression for the corresponding Jacobi operator $\jac_{x_0}$, which represents the above quadratic form and is the linearized operator of the Euler-Lagrange equations \eqref{eq:el-kahler},
\begin{equation}\label{eq:jacobioperator}
\jac_{x_0}(h)=\Delta^2 h+\delta \sigma_{\Ric^\perp(J\nabla h)}-2\delta\sigma_{B(JH,\nabla h)}-JH(JH(h)),
\end{equation}
where $\Ric^\perp(X)$ for a normal vector $X$ to $\Sigma$ is defined by $g(\Ric^\perp(X),Y)=\Ric(X,Y)$ for all $Y$ normal to $\Sigma$.

\section{The variational framework}
\label{sec:proofmain}

Let us describe the appropriate variational framework that yields the proof of the Theorem in the Introduction, as an application of the equivariant implicit function theorem, as formulated, e.g., in \cite[Thm 3.2]{bps}.

\subsection{Unparametrized embeddings}
\label{subsec:unpar}

We denote by $\embsigma$ the space of Lagrangian embeddings $x\colon\Sigma\hookrightarrow (M,\omega)$ of class $C^{k,\alpha}$, $k\geq 4$, where the regularity choice is due to Fredholmness reasons. The space $\embsigma$ endowed with the corresponding $C^{4,\alpha}$ topology is a smooth Banach manifold, and its tangent space at $x\in\embsigma$ can be identified with the space of $C^{4,\alpha}$ vector fields along $x$ that are Lagrangian variations of $x$, see Subsection~\ref{sub:LagSubmanifolds}.

There is a natural action of the diffeomorphism group $\mathrm{Diff}(\Sigma)$ on $\embsigma$. Two Lagrangian embeddings $x_i\in\embsigma$, $i=1,2$, are \emph{congruent} if there exists a diffeomorphism $\psi\colon\Sigma\to\Sigma$ such that $x_1=x_2\circ\psi$, i.e., if they belong to the same orbit of this action. Given a Lagrangian embedding $x\in\embsigma$, we denote by $[x]$ its congruence class, i.e., the orbit of $x$. Denote by $\lagsigma$ the orbit space of \emph{unparametrized} Lagrangian embeddings of $\Sigma$ in $(M,\omega)$:
\begin{equation}\label{eq:lagsigma}
\lagsigma:=\embsigma/\mathrm{Diff}(\Sigma),
\end{equation}
i.e., the set of congruence classes of Lagrangian embeddings of $\Sigma$ into $M$. In other words, an element $[x]\in\lagsigma$ is a class of embeddings of $\Sigma$ in $M$ whose elements can be obtained from one another by reparametrizations. The set $\lagsigma$ can be thus identified with the set of Lagrangian submanifolds of $M$ (of class $C^{4,\alpha}$) that are diffeomorphic to $\Sigma$. We consider $\lagsigma$ endowed with the induced quotient topology. The action of $\mathrm{Diff}(\Sigma)$ is neither free nor proper, and the orbit space $\lagsigma$ fails to be a smooth Banach manifold.

Let us briefly describe the structure of $\lagsigma$. A classic result due to Weinstein~\cite{we} states that given a smooth (i.e., $C^\infty$) Lagrangian embedding $x_0\colon\Sigma\to M$, there exists a smooth symplectomorphism $\Psi$ from a neighborhood $U$ of $x_0(\Sigma)$ in $(M,\omega)$ to a neighborhood $V$ of the zero section of the cotangent bundle $T\Sigma^*$ endowed with its canonical symplectic structure, such that $\Psi\circ x_0$ is the inclusion of the zero section into $T\Sigma^*$.
It is an easy observation that the image of a $1$-form is a Lagrangian submanifold of $T\Sigma^*$ if and only if this $1$-form is closed.
Consequently, small Lagrangian deformations of $x_0\colon \Sigma\to M$ are parametrized by closed $1$-forms on $\Sigma$.
More precisely, given any Lagrangian embedding $x\colon\Sigma\to U$ of class $C^{4,\alpha}$ sufficiently close to $x_0$, there exists a (unique) closed $1$-form $\eta_x$ on $\Sigma$ such that $\Psi\big(x(\Sigma)\big)$ is the image of $\eta_x\colon\Sigma\to T\Sigma^*$. In other words, the Lagrangian embeddings $\Psi\circ x$ and $\eta_x$ are congruent.
The map $[x]\mapsto\eta_x$ gives a continuous bijection from a neighborhood of $[x_0]$ in $\lagsigma$ to a neighborhood of the origin in the Banach space of closed $1$-forms on $\Sigma$ of class $C^{4,\alpha}$. As $x$ varies in the set of smooth Lagrangian embeddings of $\Sigma$ into $M$, such bijections form an atlas of charts for the topological manifold $\lagsigma$.

We observe, however, that the transition maps between two of these charts are in general only continuous, and not differentiable.
This is due to a subtle technicality, which in particular implies that right-composition with a diffeomorphism of class $C^{4,\alpha}$ is
not a differentiable map in the set of maps of class $C^{4,\alpha}$ between two smooth manifolds. This and other relevant issues concerning the lack of regularity of the space of unparametrized embeddings are discussed thoroughly in \cite{AliPic10}. As explained in this reference, since we are only interested in local questions around a smooth embedding, we can use the above chart as an identification and formally treat $\lagsigma$ as a smooth manifold. In this way, for convenience of notation we henceforth refer to, e.g., tangent spaces, distributions and smooth functions on $\lagsigma$, and the implicit rigorous version of these objects are the corresponding objects defined in a small neighborhood of the origin of the Banach space of closed $1$-forms on $\Sigma$ of class $C^{4,\alpha}$.

Under the above convention, if $x_0\in\embsigma$ is smooth, the tangent space at $[x_0]$ to $\lagsigma$ can be identified as
\begin{equation}\label{eq:tangent}
T_{[x_0]}\lagsigma = Z^1(\Sigma),
\end{equation}
i.e., with the Banach space $Z^1(\Sigma)$ of closed $1$-forms on $\Sigma$ of class $C^{4,\alpha}$. Note this is the image of the surjective linear map $T_{x_0}\embsigma\ni X\mapsto \sigma_X\in Z^1(\Sigma)$, whose kernel corresponds to variations tangent to $x_0$, i.e., reparametrizations, which form the tangent space to the orbit of $\mathrm{Diff}(\Sigma)$ that passes through $x_0$. This linear map is precisely the linearization of the orbit space projection $\embsigma\to\lagsigma$ at the smooth embedding $x_0$.

The tangent space \eqref{eq:tangent} has a distinguished subspace, namely $B^1(\Sigma)$, formed by exact $1$-forms on $\Sigma$ of class $C^{4,\alpha}$. This subspace corresponds to Hamiltonian variations of $x_0$, and gives rise to an integrable distribution\footnote{Since we are working locally around smooth points, consider this distribution defined in an open subset that is the domain of a chart; where it is integrable in the usual sense
that it is tangent to a foliation of this subset.} of $\lagsigma$ with codimension $b_1(\Sigma)=\dim H^1(\Sigma,\R)$, see \cite{we2}. We call this distribution the \emph{Hamiltonian distribution} in $\lagsigma$. Given a Lagrangian embedding $x_0\colon\Sigma\to M$, the integral leaves of the Hamiltonian distribution near
$[x_0]$ are parametrized by elements of the first de Rham cohomology $H^1(\Sigma,\mathds R)$. Given a closed $1$-form $\eta$ on $\Sigma$, we denote by $[\eta]$ its cohomology class and by $\lagsigma_{[\eta]}$ the integral leaf of the Hamiltonian distribution corresponding to $[\eta]$. In particular, when $\eta$ is exact, i.e., $[\eta]=0$, then $\lagsigma_{[\eta]}$ is the integral leaf through $[x_0]$, i.e., consists of all the Hamiltonian deformations of $x_0$.

\subsection{Volume functional}
We now describe how to encode the variational problem described in Subsection \ref{subsec:minvol} in the above setup, for a varying family of metrics. Namely, we start from a family of volume functionals parametrized by a family $g_t\in\Met(M)$, $t\in [-\delta,\delta]$, of metrics on $M$,
\begin{equation*}
\V\colon\embsigma\times [-\delta,\delta]\to\R,\quad \V(x,t)=\V_{g_t}\big(x(\Sigma)\big).
\end{equation*}
This functional is clearly invariant under reparametrizations, i.e., under the action of $\mathrm{Diff}(\Sigma)$. Hence, it passes to the quotient, defining a continuous map
\begin{equation}\label{eq:vol}
\V\colon\lagsigma\times [-\delta,\delta]\to\R,
\end{equation}
that is smooth in every local chart around $[x_0]\in \lagsigma$, where $x_0\in\embsigma$ is smooth, see \cite[Cor~4.4]{AliPic10}.

For any fixed $t\in[-\delta,\delta]$, the critical points of \eqref{eq:vol} are exactly the $g_t$-Lagrangian stationary Lagrangian embeddings of $\Sigma$ in $M$. Consider now a Lagrangian embedding $x_0\colon\Sigma\to M$, and
and let $\eta$ be a closed $1$-form on $\Sigma$ whose cohomology class $[\eta]$ belongs to a sufficiently small
neighborhood of $0$ in the vector space $H^1(\Sigma,\mathds R)$. For any fixed $t\in[-\delta,\delta]$, the critical points of the restriction of \eqref{eq:vol} to $\lagsigma_{[\eta]}$ are precisely the $g_t$-Hamiltonian stationary Lagrangian embeddings of $\Sigma$ in $M$ that belong to $\lagsigma_{[\eta]}$. This is the constrained variational problem that we use to prove our main result. Under the appropriate identifications, the Euler-Lagrange equation and Jacobi operator of this variational problem coincide with the ones discussed in the previous section.

Note that if there is an isometric Hamiltonian action of a Lie group $G$ on $M$, the induced $G$-action by left-composition on the space of embeddings $\embsigma$ commutes with the $\mathrm{Diff}(\Sigma)$-action and hence induces a $G$-action on $\lagsigma$. It is a straightforward observation that the Hamiltonian distribution in $\lagsigma$ is preserved by this $G$-action, and that the above constrained variational problem is also invariant under such action.

\subsection{$G$-nondegenerate embeddings}
Since the variational problem in question is $G$-invariant, the linearization of deformations that correspond to the $G$-action will automatically produce elements in the kernel of the Jacobi operator \eqref{eq:jacobioperator}. More precisely, for each $x_0$ consider the linear map
\begin{equation}\label{eq:incl}
\mathfrak g\ni X\longmapsto \sigma_{X^*}\in B^1(\Sigma),
\end{equation}
which associates to $X\in\mathfrak g$ the exact $1$-form on $\Sigma$ given by $\sigma_{X^*}=x_0^*\big(\omega(X^*,\cdot)\big)$, where $X^*$ is the action field corresponding to $X$. Such $1$-form is exact, because the $G$-action is Hamiltonian. Since the action preserves $\omega$, $g_0$ and the Hamiltonian distribution, the image $\mathcal N_{x_0}$
of the linear map \eqref{eq:incl} is contained in the kernel of $\jac_{x_0}$, since this is the second variation of the volume functional (restricted to the integral leaf $\lagsigma_{[0]}$ through $[x_0]\in\lagsigma$ of the Hamiltonian distribution). Here we are identifying the space of exact $1$-forms $B^1(\Sigma)$ of class $C^{k,\alpha}$ on $\Sigma$ with the space of real-valued functions modulo constants $C^{k,\alpha}(\Sigma)/\R$; and $\mathcal N_{x_0}$ can be identified with the tangent space at $[x_0]$ to the $G$-orbit of $[x_0]\in\lagsigma_{[0]}$.

\begin{definition}\label{def:nondeg}
The $g_0$-Hamiltonian stationary embedding $x_0$ is said to be \emph{$G$-nondegenerate} if $\mathcal N_{x_0}$ coincides
with the kernel of $\jac_{x_0}$ in $B^1(\Sigma)$.
\end{definition}

In other words, $x_0$ is $G$-nondegenerate if the kernel of $\jac_{x_0}$ is as small as it can be, since fields originating from \eqref{eq:incl} are necessarily in it. Examples of $G$-nondegenerate embeddings will be given in Section~\ref{sec:examples}.

\subsection{Proof of main result}
We are now ready for the proof of the Theorem in the Introduction. For convenience, we restate it below in the language of compatible triples and unparametrized embeddings.

\begin{theorem*}
Let $(M,\omega,g_0,J_0)$ be a K\"ahler manifold with an isometric Hamiltonian action of a compact Lie group $G$. Let $(\omega,g_t,J_t)$ be a deformation of the compatible triple $(\omega,g_0,J_0)$, such that $G$ acts by $g_t$-isometries and $J_t$-biholomorphisms, for all $t\in [-\delta,\delta]$. Suppose $x_0\colon\Sigma\to (M,\omega)$ is a $G$-nondegenerate $g_0$-Hamiltonian stationary Lagrangian embedding. Then, there exists $\varepsilon>0$, a neighborhood $\mathcal V$ of $[x_0]\in\lagsigma$, a neighborhood $\mathcal E$ of $[0]\in H^1(\Sigma,\mathds R)$ and a map $x\colon\, (-\varepsilon,\varepsilon)\times\mathcal E\,\to\mathcal V$; such that $x(0,0)=[x_0]$ and $x(t, [\eta])$ is a $g_t$-Hamiltonian stationary Lagrangian unparametrized embedding in $\lagsigma_{[\eta]}$, for all $|t|<\varepsilon$ and all $[\eta]\in\mathcal E$. Moreover, given $(t,[\eta])\in (-\varepsilon,\varepsilon)\times\mathcal E$, if $[x_*]\in\mathcal V$ is a $g_t$-Hamiltonian stationary Lagrangian unparametrized embedding in $\lagsigma_{[\eta]}$ sufficiently close to the $G$-orbit of $[x_0]$, then there exists $\phi\in G$ such that $\phi([x_*])=x(t,[\eta])$.
\end{theorem*}

\begin{proof}
By assumption, the $G$-action preserves $\omega$ and $g_t$, for all $t\in [-\delta,\delta]$. Thus, the induced $G$-action by left-composition on the space of embeddings $x\colon\Sigma\to M$ preserves Lagrangian embeddings, as well as their $g_t$-volume. Moreover, since the action is assumed Hamiltonian, it also preserves the leaves of the Hamiltonian distribution. This means that by choosing a $G$-nondegenerate $g_0$-Hamiltonian stationary Lagrangian embedding $x_0$ we are in the setup of the $G$-equivariant implicit function theorem studied in \cite{bps}.

Using identification \eqref{eq:tangent} and the canonical splitting $Z^1(\Sigma)=B^1(\Sigma)\oplus H^1(\Sigma,\R)$, we can write a sufficiently small neighborhood $\mathcal U$ of $[x_0]\in\lagsigma$ as a product $\mathcal{U}=\mathcal{U}_B\times\mathcal{U}_H$, where $\mathcal{U}_B$ is a neighborhood of $0\in B^1(\Sigma)$ and  $\mathcal{U}_H$ is a neighborhood of $[0]\in H^1(\Sigma)$. In this way, any $[x]\in\mathcal U$ corresponds to a unique pair $(\beta,[\eta])\in\mathcal{U}_B\times\mathcal{U}_H$, and $[x_0]$ corresponds to $(0,[0])$. Moreover, each slice $\mathcal L_{[\eta]}:=\{(\beta,[\eta]):\beta\in\mathcal{U}_B\}$ of $\mathcal U_B\times\mathcal U_H$ corresponds to the intersection of the leaf $\lagsigma_{[\eta]}$ with $\mathcal U$.
 The abstract implicit function theorem is applied to the volume functional $\V$ in \eqref{eq:vol}, considered as a function of three variables in the neighborhood $\mathcal U$ of $[x_0]$; $\V(t,\beta,[\eta])$, where $t$ varies in $[-\delta,\delta]$, $\beta$ varies in $\mathcal{U}_B$ and $[\eta]$ varies in $\mathcal U_H$. Thus, $g_t$-Hamiltonian stationary Lagrangian embeddings correspond to points where the derivative of $\V(t,\beta,[\eta])$ with respect to $\beta$ vanishes.

In this setup, the only hypothesis in the equivariant implicit function theorem that requires additional explanation is the Fredholmness of the Jacobi operator $\jac_{x_0}$, which corresponds to the second variation of $\V$ with respect to the variable $\beta$ at the point $(0,0,[0])$. This operator is defined on the tangent space to $\mathcal U_B$ at the origin, which is $B^1(\Sigma)$, that we now write as $B_{k,\alpha}^1(\Sigma)$ to emphasize that its elements are exact $1$-forms on $\Sigma$ \emph{of class $C^{k,\alpha}$}. Such linear space is canonically identified with $C^{k,\alpha}(\Sigma)/\R$, the real-valued functions H\"older space $C^{k,\alpha}$ modulo constants. Although so far we were implicitly using this identification for convenience, it is now important to write it explicitly. The Jacobi operator of the variational problem above mentioned is given by the composition of the linear maps in the bottom line of the following commutative diagram:
\begin{equation*}
\xymatrix@+14pt{ & C^{k,\alpha}(\Sigma) \ar[r]^{\jac_{x_0}} \ar[d] & C^{k-4,\alpha}(\Sigma) \ar[d] & \\
B^1_{k,\alpha}(\Sigma) \ar[r]^\cong & C^{k,\alpha}(\Sigma)/\R \ar[r]^{[\jac_{x_0}]} \ar@{.>}[ru] & C^{k-4,\alpha}(\Sigma)/\R \ar[r]^\cong &  B^1_{k-4,\alpha}(\Sigma),
}
\end{equation*}
where the vertical arrows are the natural projections and the top line operator $\jac_{x_0}$ is given by formula \eqref{eq:jacobioperator}. This formula shows that $\jac_{x_0}$ is a fourth-order (formally self-adjoint) linear elliptic operator. Thus, from standard Schauder estimates, such an operator is Fredholm of index $0$. Since constant functions are in the kernel of $\jac_{x_0}$, it induces an operator from the quotient $C^{k,\alpha}(\Sigma)/\R$, denoted by the dotted arrow. Such operator the has same image as $\jac_{x_0}$ and its kernel is $(\ker\jac_{x_0})/\R$. Hence, it is a Fredholm operator of index $-1$. The Jacobi operator $[\jac_{x_0}]\colon C^{k,\alpha}(\Sigma)/\R\to C^{k-4,\alpha}(\Sigma)/\R$ is given by the composition of the latter with the projection $C^{k-4,\alpha}(\Sigma)\to C^{k-4,\alpha}(\Sigma)/\R$, which is a Fredholm operator of index $+1$. Thus, $[\jac_{x_0}]$ is also Fredholm and its index is the sum of the indices of its factors, which is $0$. This proves the desired Fredholmness condition.
From\footnote{For further details, the reader may follow the proof of the constant mean curvature hypersurfaces version of the equivariant implicit function theorem discussed in \cite{bps}, which has many analogies with the application discussed here.} \cite[Thm 3.2]{bps}, we now get that there exists a map $\beta(t,[\eta])$ (defined locally) such that the map $x(t,[\eta]):=(\beta(t,[\eta]),[\eta])\in\mathcal U_B\times\mathcal U_H=\mathcal U\subset\lagsigma$ satisfies the desired conditions.
\end{proof}

\section{Examples of deformations}
\label{sec:examples}

In this section, we describe some deformations of compatible triples, and a few examples to which the result proved above applies.

\subsection{Cheeger deformations}
We now briefly outline an important example of metric deformation preserving symmetries, the so-called \emph{Cheeger deformation}, and the corresponding deformation of almost complex structures via Corollary~\ref{cor:deform}. Cheeger deformations are very important tools in Riemannian geometry (see \cite{zillermueter,zillersurvey}), and its counterpart in almost complex manifolds given by the above correspondence apparently has not yet been thoroughly explored.

Let $J_0\in\J$ be an $\omega$-compatible almost complex structure on $(M,\omega)$, and $g_0(\cdot,\cdot)=\omega(\cdot,J_0\cdot)$. Suppose $G$ is a compact Lie group of symplectomorphisms of $(M,\omega)$, that acts on $(M,g_0)$ by isometries. This gives a partition of $M$ into $G$-orbits, and the deformation $g_t$ of $g_0$ we now describe essentially works by rescaling $g_0$ in the direction of these orbits, leaving it unchanged in the complementary directions. This will then automatically cause $g_t$ to be $G$-invariant, i.e., the deformation will be through metrics that still have an isometric $G$-action. Such deformations have been extensively used in many situations, see \cite{zillermueter,zillersurvey}, and we outline its construction following the notation of the above references.

Fix a bi-invariant metric $Q$ on $G$, and consider the product manifold $M\times G$ endowed with the product metric $g+\tfrac1tQ$. Denote by $g\cdot p$ the action of $g\in G$ on $p\in M$ and define a submersion
\begin{equation*}
\rho\colon M\times G\to M, \quad \rho(p,g)=(g^{-1})\cdot p.
\end{equation*}
Let $g_t\in\Met(M)$ be the unique metric that makes $\rho$ a Riemannian submersion. It is immediate from its definition that $g_t$ is $G$-invariant, i.e., $G$ acts isometrically on $(M,g_t)$, $t>0$. The curve of metrics $g_t$, $t>0$, extends smoothly across $t=0$, and coincides with the original metric $g_0$ at this point. In this sense, $g_t$ is a deformation of $g_0$.

In order to analyze how $g_t$, $t>0$, differs from $g_0$, we have to introduce some more notation. Let $G_p$ be the isotropy group at $p\in M$ and $\mathfrak g_p$ its Lie algebra. Fix the $Q$-orthogonal decomposition $\mathfrak g=\mathfrak g_p\oplus\mathfrak m_p$, and identify $\mathfrak m_p$ with the tangent space $T_pG(p)$ to the $G$-orbit through $p$, via action fields; i.e., $X\in\mathfrak m_p$ is identified with $X^*_p=\frac{\dd}{\dd s}\exp(sX)|_{s=0}\in T_pG(p)$. This induces a $g_t$-orthogonal decomposition $T_pM=\mathcal V_p\oplus\mathcal H_p$ in vertical space $\mathcal V_p=\{X_p^*\in T_pG(p):X\in\mathfrak m_p\}$ and horizontal space $\mathcal H_p=\mathcal V_p^{\perp}$, where $^\perp$ is the $g_t$-orthogonal complement. Let
\begin{equation*}
P_t\colon\mathfrak m_p\to\mathfrak m_p, \quad Q(P_t(X),Y)=g_t(X^*_p,Y^*_p).
\end{equation*}
Then $P_t$ is a $Q$-symmetric automorphism that represents $g_t$ in terms of $Q$, and it can be easily computed (see \cite[Prop 1.1]{zillermueter}) that $P_t=P_0\,(\id+tP_0)^{-1}$, $t\geq0$. Thus, defining
\begin{equation*}
C_t\colon T_pM\to T_pM, \quad g(C_t(X),Y)=g_t(X,Y),
\end{equation*}
we get
\begin{equation*}
C_t(X)=P_0^{-1}P_t(X^\mathcal V)+X^\mathcal H,
\end{equation*}
where $X^\mathcal V$ and $X^\mathcal H$ are the vertical and horizontal components of $X$ respectively. If $P_0$ has eigenvalues $\lambda_i$, then by the above formula, $C_t$ has eigenvalues $\frac{1}{1+t\lambda_i}$ in the vertical directions and $1$ in the horizontal directions. This means that as $t$ increases, the metric $g_t$ shrinks in the direction of the $G$-orbits and stays unchanged in the remaining directions.

By Corollary~\ref{cor:deform}, given the above deformation $g_t$ of $g_0$, there is a corresponding deformation $J_t=J_{g_t}$ of $J_0$, such that the $G$-action by symplectomorphisms on $M$ is also by $g_t$-isometries and $J_t$-biholomorphisms, for $t>0$.

\subsection{Minimal Lagrangians in K\"ahler-Einstein manifolds}
Hamiltonian stationary Lagrangians in K\"ahler-Einstein manifolds have been extensively studied in the literature, see e.g. \cite{ohnita,ohnita2,bc,lee,oh1,oh2,pp}. In this subsection, we are interested in the particular case of \emph{minimal} Lagrangians $\Sigma$ of these manifolds, which are automatically Hamiltonian stationary, see Subsection~\ref{subsec:minvol}. In this case, the Jacobi operator \eqref{eq:jacobioperator} assumes a very simple form. Namely, if $\kappa$ is the Einstein constant of $(M,g_0)$, i.e., $\Ric=\kappa\,g_0$ and $x_0\colon\Sigma\to M$ is a minimal Lagrangian submanifold, then
\begin{equation}
\jac_{x_0}(h)=\Delta^2 h+\kappa\,\delta \sigma_{J\nabla h}=\Delta(\Delta h-\kappa \, h),
\end{equation}
since $\sigma_{J\nabla h}=-\dd h$. Let us analyze the kernel of this operator. Since $\Sigma$ is compact, the function $\Delta h-\kappa \, h$ is harmonic if and only if it is constant. We are working on $C^{k,\alpha}(\Sigma)/\R$, i.e., modulo constants, so it follows that elements in the kernel of the Jacobi operator are precisely the exact $1$-forms $\dd h$ such that
\begin{equation}\label{eq:eigenvalue}
\Delta h-\kappa\,h=0,
\end{equation}
i.e., $h$ is an eigenfunction of the Laplacian on $\Sigma$, corresponding to the Einstein constant $\kappa$ of $M$.

The \emph{Hamiltonian stability} of such minimal Lagrangian submanifolds, i.e., whether their Jacobi operator is nonnegative, plays an important role in the theory. An immediate conclusion from the above is that $\Sigma$ is Hamiltonian stable if and only if $\lambda_1(\Sigma)\geq\kappa$, where $\lambda_1(\Sigma)$ is the first eigenvalue of the Laplacian on (functions on) $\Sigma$, cf. \cite[Thm 4.4]{oh1}.

For our equivariant setup, the K\"ahler-Einstein manifold $M$ has an isometric Hamiltonian $G$-action, hence every action field not tangent to $\Sigma$ induces an element of its Jacobi operator through the map \eqref{eq:incl}. In particular, \eqref{eq:eigenvalue} always has nontrivial solutions, i.e., the Einstein constant $\kappa$ is an eigenvalue of the Laplacian on $\Sigma$. Such a minimal Lagrangian is $G$-nondegenerate if and only if all solutions $h$ of \eqref{eq:eigenvalue} are of this form, i.e., if there exists an action field $X^*$ such that $h=\langle\mu(\cdot),X^*\rangle|_\Sigma$, where $\mu$ is the moment map of the action. For example, one way to verify this condition is by dimensional reasons. Namely, if the dimension of the span of action fields normal to $x_0(\Sigma)\subset M$ is larger than or equal to the multiplicity of $\kappa$ as an eigenvalue of the Laplacian on $\Sigma$, then $\Sigma$ is $G$-nondegenerate. This provides a setup to which our main result applies, considering the $1$-parameter family of compatible triples $(\omega,J_t,g_t)$ obtained by a Cheeger deformation of $g_0$ with respect to the $G$-action.

Let us give a few concrete examples in the case $(M,g_0)$ is $\C P^n$ with its standard K\"ahler structure, for which $\kappa=2n+2$. More precisely, we will consider totally real minimal submanifolds $x_0\colon\Sigma\to \C P^n$ with parallel second fundamental form. These were classified by Naitoh and Takeuchi, see \cite[Sec 2]{ohnita}. Following the classification, Amarzaya  and Ohnita \cite{ohnita} determined which of those submanifolds are Hamiltonian stable. Namely, they obtained the following, see \cite[Thm 4.1]{ohnita}.

\begin{theorem}[Amarzaya-Ohnita]\label{thm:ohnita}
Let $\Sigma$ be a $n$-dimensional totally real minimal submanifold embedded in $\C P^n$ with parallel second fundamental form in the following table:
\begin{center}
\begin{tabular*}{0.5\textwidth}{@{\extracolsep{\fill}} c c }
\hline
$\Sigma$ & $n$ \rule[-1.2ex]{0pt}{0pt} \rule{0pt}{2.5ex} \\
\hline
\hline
$\SU(p)/\Z_p$  \rule[-1.2ex]{0pt}{0pt} \rule{0pt}{2.5ex}  & $p^2-1$  \\
$\SU(p)/\SO(p)\Z_p$  \rule[-1.2ex]{0pt}{0pt} \rule{0pt}{2.5ex} & $(p-1)(p+2)/2$ \\
 $\SU(2p)/\mathsf{Sp}(p)\Z_{2p}$  \rule[-1.2ex]{0pt}{0pt} \rule{0pt}{2.5ex} & $(p-1)(2p+1)$ \\
$\mathsf E_6/\mathsf F_4\Z_3$   \rule[-1.2ex]{0pt}{0pt} \rule{0pt}{2.5ex}  & $26$\\
\hline
\hline
\end{tabular*}
\end{center}
Then $\Sigma$ is a Hamiltonian stable minimal Lagrangian submanifold in $\C P^n$. Moreover, the kernel of the Jacobi operator of $\Sigma$ is exactly the span of the normal projections of Killing vector fields on $\C P^n$.
\end{theorem}

With the above result at hand, one can apply our deformation methods to a such $\Sigma$ using any Cheeger deformation of $\C P^n$ with respect to a $G$-action that has the following extra property: the normal space at each $p\in x_0(\Sigma)$ must be contained in the tangent space $T_p G(p)$ to the $G$-orbit through $p$. It then automatically follows that the image of the map \eqref{eq:incl} is precisely the kernel of the Jacobi operator $\jac_{x_0}$, i.e., $x_0$ is $G$-nondegenerate.

\subsection{Sasaki metrics on tangent bundles}
Let $(N,g)$ be a Riemannian manifold, and consider its tangent bundle $M=TN$. We recall a standard construction of a Riemannian metric on $M$, starting from a metric on $N$, see also \cite{tanbdls}. Let $\pi\colon TN\to N$ denote the canonical projection; for $v\in TN$, write $T_v(TN)=\mathrm{Ver}_v\oplus\mathrm{Hor}_v$, where $\mathrm{Ver}_v$ is the \emph{vertical subspace}, i.e., the tangent space to the fiber $T_pN$, where $p=\pi(v)$, and $\mathrm{Hor}_v$ is the
\emph{horizontal subspace} determined by the Levi-Civita connection of $g$. Given $\xi\in T_v(TN)$,
we denote by $\xi^\mathrm{ver}$ and $\xi^\mathrm{hor}$ its vertical and horizontal component, respectively.

The spaces $\mathrm{Ver}_v$ and $\mathrm{Hor}_v$ are $g_\mathrm S$-orthogonal.
There is a canonical isomorphism $\mathrm{Ver}_v=T_v(T_pN)\to T_pN$; the restriction of $g_\mathrm S$ to $\mathrm{Ver}_v$ is defined to be equal to the pull-back of $g_p$ through such isomorphism.
Moreover, the restriction of the differential $\mathrm d\pi(p)\big\vert_{\mathrm{Hor}_v}\colon\mathrm{Hor}_v\to T_pN$
is an isomorphism, the restriction of $g_\mathrm S$ to $\mathrm{Hor}_v$ is defined to be equal to the pull-back of $g_p$ through such isomorphism. This defines a smooth Riemannian tensor $g_\mathrm S$ on $M=TN$, called the \emph{Sasaki metric associated to $g$}.

In addition, a symplectic form $\omega_g$ can be defined on $M$, as follows. Given $v\in TN$, and $\xi,\eta\in T_v(TN)$, let
\begin{equation*}
\omega_g(\xi,\eta):=g(\xi^\mathrm{ver},\eta^\mathrm{hor})-g(\xi^\mathrm{hor},\eta^\mathrm{ver}).
\end{equation*}
This symplectic structure interacts well with the Sasaki metric, due to the following observation.

\begin{lemma}\label{thm:dfisometry}
Let $f\colon N\to N$ be a $g$-isometry. Then, $\dd f\colon TN\to TN$ preserves both the Sasaki metric $g_\mathrm S$ and the symplectic form $\omega_g$. Moreover, the map
\begin{equation*}
\Iso(N,g)\ni f\mapsto\mathrm df\in\Iso(TN,g_\mathrm S)\cap \Symp(TN,\omega_g)
\end{equation*}
is an injective Lie group homomorphism with closed image.\qed
\end{lemma}

In particular, if $G$ is a Lie group acting by isometries on a Riemannian manifold $(N,g)$, we also have a $G$-action on $M$ by $g_\mathrm S$-isometries that preserve $\omega_g$. If the $G$-action on $N$ preserves a $1$-parameter family of Riemannian metrics $g_t$ (e.g., a Cheeger deformation), then the corresponding action on $M$ provides an example of the situation considered in the Theorem in the Introduction, with the choice of a $G$-nondegenerate Hamiltonian stationary Lagrangian of $M$.

\subsection{K\"ahler-Ricci flow}\label{subsec:ricciflow}
One situation in which the deformation $(\omega,J_t,g_t)$ preserves the integrability of $J_t$, i.e., the fact that $(M,\omega,J_t,g_t)$ is K\"ahler, is when $g_t$ evolves by the K\"ahler-Ricci flow. This means $g_t$ is a solution of the evolution equation $\frac{\partial}{\partial t}g_t=-2\Ric(g_t)$, which also clearly preserves the isometries of the initial metric $g_0$. Thus, if there is a Hamiltonian isometric action of $G$ on $(M,g_0)$, we automatically get that $G$ acts on $M$ by $g_t$-isometries and $J_t$-biholomorphisms. In this way, $G$-nondegenerate $g_0$-Hamiltonian stationary Lagrangians of $(M,\omega,J_0,g_0)$ may be deformed to $g_t$-Hamiltonian stationary Lagrangians up to the $G$-action, where $g_t$ is the K\"ahler-Ricci flow of $g_0$.

\end{document}